\documentclass[a4paper,10pt,psamsfonts]{amsart}

\usepackage{amsmath,amsfonts,amssymb}
\usepackage{xypic}
\usepackage[english]{babel}
\usepackage{color}
\usepackage{url}
\usepackage{fullpage}

\newtheorem{thm}{Theorem}[section]
\newtheorem{prop}[thm]{Proposition}

\newtheorem{lemma}[thm]{Lemma}
\theoremstyle{definition}
\newtheorem{defi}[thm]{Definition}
\newtheorem{rem}[thm]{Remark}
\newtheorem{ex}[thm]{Example}

\newcommand{\N}{\mathbb{N}}
\newcommand{\Z}{\mathbb{Z}}

\newcommand{\R}{\mathbb{R}}
\newcommand{\C}{\mathbb{C}}
\newcommand{\T}{\mathbb{T}}
\newcommand{\st}{\;:\;}
\newcommand{\definizione}{\;\stackrel{\text{def}}{=}\;}

\newcommand{\sspace}{\cdot}
\newcommand{\ssspace}{\cdot}

\newcommand{\opiccolo}[1]{\mathrm{o}\left[#1\right]}

\DeclareMathOperator{\im}{i}

\DeclareMathOperator{\de}{d}
\DeclareMathOperator{\id}{id}

\DeclareMathOperator{\interno}{int}

\newcommand{\Cpf}{$\mathcal{C}^\infty$-pure-and-full}

\newcommand{\Cf}{$\mathcal{C}^\infty$-full}
\newcommand{\Cp}{$\mathcal{C}^\infty$-pure}
\newcommand{\pf}{pure-and-full}

\newcommand{\p}{pure}
\newcommand{\del}{\partial}
\newcommand{\delbar}{\overline{\del}}

\title[On the cohomology of almost-complex manifolds]
{On the cohomology of almost-complex manifolds}
\author{Daniele Angella}
\address[Daniele Angella]{Dipartimento di Matematica ``Leonida Tonelli''\\
Universit\`{a} di Pisa \\
Largo Bruno Pontecorvo 5, 56127\\
Pisa, Italy}
\email{angella@mail.dm.unipi.it}
\author{Adriano Tomassini}
\address[Adriano Tomassini]{Dipartimento di Matematica\\
Universit\`{a} di Parma \\
Parco Area delle Scienze 53/A, 43124 \\
Parma, Italy}
\email{adriano.tomassini@unipr.it}

\thanks{This work was supported by the Project MIUR ``Geometric Properties of Real and Complex Manifolds'' and by GNSAGA
of INdAM}

\begin{document}

\vspace{-2cm}
\begin{minipage}[l]{10cm}
{\sffamily
  D. Angella, A. Tomassini, On the cohomology of almost-complex manifolds,
  {\em Int. J. Math.} \textbf{23} (2012), no.~2, 1250019, 25 pp.,
  \textsc{doi:} \texttt{10.1142/S0129167X11007604}.

\medskip

  \begin{flushright}\begin{footnotesize}
  (Electronic version of an article published as {\em Int. J. Math.}, \textbf{23}, no.~2, 2012, 1250019, 25 pp., \textsc{doi:} \texttt{10.1142/S0129167X11007604}\\
  \textcopyright\ [copyright World Scientific Publishing Company] \url{http://www.worldscientific.com/worldscinet/ijm}.)
  \end{footnotesize}\end{flushright}
}
\end{minipage}
\vspace{2cm}

\begin{abstract}
 Following T.-J. Li, W. Zhang \cite{li-zhang}, we continue to study the link between the cohomology of an almost-complex manifold and its almost-complex structure. In particular, we apply the same argument in \cite{li-zhang} and the results obtained by D. Sullivan in \cite{sullivan} to study the cone of semi-K\"ahler structures on a compact semi-K\"ahler manifold.
\end{abstract}

\maketitle
\tableofcontents
\section*{Introduction}
In studying differentiable manifolds and their structures, the de Rham and Dolbeault cohomologies provide an important tool. On a compact complex manifold $M$, the Hodge and Fr\"olicher spectral sequence relates the latter to the former, as follows
$$ E_{1}^{p,q}\simeq H^{p,q}_{\delbar}(M) \Rightarrow H^{p+q}_{dR}(M;\C) \;.$$
On a compact K\"ahler manifold as well as on a compact complex surface, this spectral sequence degenerates at the first level and the trivial filtration on the space of differential forms induces a weight $2$ formal Hodge decomposition in cohomology: that is, the de Rham cohomology admits a decomposition through the Dolbeault cohomology; in other words, the complex structure yields a decomposition not only at the level of differential forms but also in cohomology. T.-J. Li and W. Zhang studied in \cite{li-zhang} the class of the almost-complex manifolds for which a similar decomposition holds, the {\em \Cpf\ manifolds} (see \S\ref{sec:cpf} for the precise definitions). Obviously, compact K\"ahler manifolds and compact complex surfaces are \Cpf; moreover, T. Dr\v{a}ghici, T.-J. Li and W. Zhang proved in \cite{draghici-li-zhang} that every $4$-dimensional compact almost-complex manifold is \Cpf.

In \cite{li-zhang}, the notion of \Cpf\ almost-complex structures arises in the study of the symplectic cones of an almost-complex manifold: more precisely, \cite[Proposition 3.1]{li-zhang} (which we quote in Theorem \ref{thm:coni-li-zhang}) proves that, if $J$ is an almost-complex structure on a compact almost-K\"ahler manifold such that every cohomology class can be written as a sum of a $J$-invariant class and a $J$-anti-invariant class (we will call such a $J$ a \Cf\ structure, see \S\ref{sec:cpf} for the precise definition), then the class of each symplectic form which is positive on the $J$-lines is the sum of the classes of an almost-K\"ahler form and of a $J$-anti-invariant one.

Other properties concerning \Cpf\ manifolds and their dual equivalents, namely the \pf\ manifolds, were studied by A. Fino and the second author in \cite{fino-tomassini}. For example, they proved that, given a compact almost-K\"ahler manifold whose symplectic form satisfies the Hard Lefschetz Condition, henceforth denoted \textsc{hlc} (that is, the symplectic form and all its powers induce isomorphisms in cohomology, see \S\ref{sec:Cpf-and-pf} for the precise definition), if it is \Cpf\ then it is also \pf; in other words, the \textsc{hlc} provides a tool for the duality between the notions of a \Cpf\ and a \pf\ structure. They also showed that the notion of \Cpf ness is not trivial in dimension greater than $4$: an explicit example of a $6$-dimensional non-\Cp\ almost-complex manifold is in \cite[Example 3.3]{fino-tomassini} (actually, it is not \Cf, too, see Example \ref{ex:fino-tomassini}); more examples of non-\Cpf\ complex structures arise from deformations of the Iwasawa manifold.

In \S\ref{sec:cpf}, we recall the basic notions introduced in \cite{li-zhang} and we give some further examples. Since many related notions are introduced, several questions arise about the connections between them. We show that being \Cp\ and being \Cf\ are not related properties; we also study $4$-dimensional non-integrable almost-complex manifolds (while they are \Cpf\ because of the strong result in \cite{draghici-li-zhang}, there may or may not be a type-decomposition at the level of $H^1_{dR}$, see Proposition \ref{prop:complex-Cpf-4}).

In \S\ref{sec:balanced}, we use the same argument as in \cite{li-zhang} (where it is proved that, if $J$ is a \Cf\ almost-K\"ahler structure, then $\mathcal{K}^t_J=\mathcal{K}^c_J+H^{-}_J(M)$) to study the cone of semi-K\"ahler structures. Recall that a semi-K\"ahler structure on a $2n$-dimensional almost-complex manifold is given by a $2$-form $\omega$ which is the fundamental $2$-form of a Hermitian metric and satisfies $\de\left(\omega^{n-1}\right)=0$. We will introduce two cones:
\begin{eqnarray*}
\mathcal{K}b^t_J &:=& \Big\{\left[\Phi\right] \in H^{2n-2}_{dR}(M;\R) \st \Phi \text{ is a }\de\text{-closed real }(2n-2)\text{-form}\\[5pt]
&& \text{that is positive on the complex }(n-1)\text{-subspaces}\Big\} \subseteq H^{2n-2}_{dR}(M;\R)
\end{eqnarray*}
and
\begin{eqnarray*}
\mathcal{K}b^c_J &:=& \Big\{\left[\Phi\right] \in H^{2n-2}_{dR}(M;\R) \st \Phi \text{ is a }\de\text{-closed real }(n-1,n-1)\text{-form}\\[5pt]
&& \text{that is positive on the complex }(n-1)\text{-subspaces}\Big\}\\[5pt]
&=& \left\{[\varphi^{n-1}] \st \varphi\text{ is a semi-K\"ahler form}\right\}\subseteq H^{(n-1,n-1)}_{J}(M)_\R \;.
\end{eqnarray*}
In the same spirit of \cite{li-zhang} and using the techniques in \cite{sullivan}, we prove the following.\\

{\noindent\bfseries Theorem \ref{thm:coni-semiK}.}{\itshape\
 Let $(M,J)$ be a $2n$-dimensional compact almost-complex manifold. Assume
 that $\mathcal{K}b^c_J\neq\varnothing$ (that is, $(M,J)$
 is semi-K\"ahler) and that $0\not\in\mathcal{K}b^t_J$. Then
\begin{equation}\tag{\ref{eq:cono-1}}
\mathcal{K}b^t_J\cap H^{(n-1,n-1)}_J(M)_\R = \mathcal{K}b^c_J
\end{equation}
and
\begin{equation}\tag{\ref{eq:cono-2}}
\mathcal{K}b^c_J+H^{(n,n-2),(n-2,n)}_J(M)_\R \subseteq \mathcal{K}b^t_J \;.
\end{equation}
Moreover, if $J$ is \Cf\ at the $(2n-2)$-th stage, then
\begin{equation}\tag{\ref{eq:cono-3}}
\mathcal{K}b^c_J+H^{(n,n-2),(n-2,n)}_J(M)_\R = \mathcal{K}b^t_J \;.
\end{equation}
}

\noindent  Observe that we have to suppose that no classes in $\mathcal{K}b^t_J$ are zero. This is clearly true for $\mathcal{K}^t_J$.

In \S\ref{sec:Cpf-and-pf}, we prove a result similar to \cite[Theorem 4.1]{fino-tomassini} by A. Fino and the second author: we will prove that a semi-K\"ahler manifold $\left(M,J,\omega\right)$ which is complex-\Cpf\ at the first stage and whose $\omega$ induces an isomorphism
$$
\omega^{n-1}\colon H^1_{dR}(M;\R)\stackrel{\simeq}{\longrightarrow}H^{2n-1}_{dR}(M;\R)
$$
is also complex-\pf\ at the first stage; moreover, under these hypotheses, we have the duality $H^{(1,0)}_J(M)\simeq H^J_{(0,1)}(M)$. The required hypothesis is not too restrictive nor trivially satisfied, as we show in some examples; however, note that it is an open property.

In \S\ref{sec:semicontinuity}, we consider the problem of the semi-continuity of $h^+_{J_t}(M):=\dim_\R H^+_{J_t}(M)$ and of $h^-_{J_t}(M):=\dim_\R H^-_{J_t}(M)$, where $\left\{J_t\right\}_t$ is a curve of almost-complex structures on a compact manifold $M$. T. Dr\v{a}ghici, T.-J. Li and W. Zhang proved that, on a $4$-dimensional compact manifold $M^4$ (where all the $J_t$ are \Cpf), $h^+_{J_t}(M^4)$ is a lower-semi-continuous function in $t$ and $h^-_{J_t}(M^4)$ is upper-semi-continuous. We give some examples showing that the situation is more complicated in dimension greater than four: in Example \ref{ex:etabeta} we present a curve of almost-complex structures on the manifold $\eta\beta_5$ as a counterexample to the upper-semi-continuity of $h^-_{J_t}(\eta\beta_5)$ and in Example \ref{ex:s3t3} we consider $\mathbb{S}^3\times\T^3$ as a counterexample to the lower-semi-continuity of $h^+_{J_t}(\mathbb{S}^3\times\T^3)$; however, note that the \Cp ness does not hold for all the structures of the curves in 
these examples, therefore one could ask for more fulfilling counterexamples. Since the problem of the semi-continuity of $h^+_{J_t}(M)$ is related to finding $J_t$-compatible symplectic structures near a $J_0$-compatible symplectic form on $M$ for $t$ small enough, we study the problem of finding obstructions to the existence of classes in $H^+_{J_t}(M)$ each of which has a $J_t$-invariant representative close to a $J_0$-invariant one. Semi-continuity here is rarely satisfied, as Proposition \ref{prop:scs-forte} shows.

\section{$\mathcal{C}^\infty$-pure-and-full almost-complex structures}\label{sec:cpf}
Let $M$ be a differentiable manifold of dimension $2n$ and let $J$ be an almost-complex structure on $M$, that is, an
endomorphism of $TM$ such that $J^2=-\id_{TM}$.

Following \cite{li-zhang}, define the following subgroups of $H^\bullet_{dR}(M;\,\C)$ and $H^\bullet_{dR}(M;\,\R)$:
if $S$ is a set of pairs $(p,q)\in\N^2$, let $H^S_J(M)$ be
$$
H^S_J(M) \;\definizione\; \left\{\left[\alpha\right]\in H^\bullet_{dR}(M;\C) \st
\alpha\in\left(\bigoplus_{(p,q)\in S}\wedge^{p,q}M\right)\,\cap\,\ker\de \right\}
$$
and
$$
H^S_J(M)_\R \;\definizione\; H^S_J(M) \,\cap\, H^\bullet_{dR}(M;\R) \;.
$$
We are mainly interested in
$$ H^+_J(M) \;:=\; H^{(1,1)}_J(M)_\R \;, \qquad H^-_J(M) \;:=\; H^{(2,0),(0,2)}_J(M)_\R \;.$$

Using the space of currents $\mathcal{D}'_\bullet(M)$ instead of the space of differential forms $\wedge^\bullet M$ and
the de Rham homology $H_\bullet(M;\R)$ instead of the de Rham cohomology $H^\bullet_{dR}(M;\R)$, we can define the subgroups
$H^J_{S}(M)$ and $H^J_{S}(M)_\R$ similarly.
\smallskip

T.-J. Li and W. Zhang gave the following.
\begin{defi}[{\cite[Definition 2.4, Definition 2.5, Lemma 2.6]{li-zhang}}]\label{def:Cpf}
 An almost-complex structure $J$ on $M$ is said:
\begin{itemize}
 \item \emph{\Cp} if
$$ H^{(2,0),(0,2)}_J(M)_\R\;\cap\; H^{(1,1)}_J(M)_\R \;=\; \left\{\left[0\right]\right\} \;; $$
 \item \emph{\Cf} if
$$
H^{(2,0),(0,2)}_J(M)_\R\;+\; H^{(1,1)}_J(M)_\R \;=\; H^2_{dR}(M;\R) \;;
$$
 \item \emph{\Cpf} if it is both \Cp\ and \Cf, i.e. if the following decomposition holds:
$$
H^{(2,0),(0,2)}_J(M)_\R\,\oplus\, H^{(1,1)}_J(M)_\R \,=\, H^2_{dR}(M;\R) \;.
$$
\end{itemize}
For a complex manifold $M$, by saying that $M$ is, for example, \Cpf, we mean that its naturally associated  integrable
almost-complex structure is \Cpf.

Using currents instead of forms, the definition of {\emph \pf} structure is recovered. Furthermore, we will say that an almost-complex manifold $\left(M,J\right)$ is \emph{\Cpf\ at the $k$-th stage} if the decomposition
$$ H^k_{dR}(M;\R) = \bigoplus_{p+q=k}H^{(p,q),(q,p)}_J(M)_\R $$
holds, and \emph{complex-\Cpf\ at the $k$-th stage} if the decomposition
$$ H^k_{dR}(M;\C) = \bigoplus_{p+q=k}H^{(p,q)}_J(M) $$
holds; again, similar notions for the currents will be understood.
\end{defi}

We refer to \cite{li-zhang}, \cite{draghici-li-zhang}, \cite{fino-tomassini} and \cite{angella-tomassini} for a more
detailed study of these and other properties. We recall that every compact complex surface and every compact K\"ahler
manifold is \Cpf\ (see \cite{draghici-li-zhang}), since, in these cases, the Hodge-Fr\"olicher spectral sequence degenerates at the first
level and the trivial filtration on the space of differential forms induces a Hodge structure of weight $2$ on $H^2_{dR}$, see, e.g.,
\cite{barth-peters-vandeven}, \cite{deligne} (in fact, a compact K\"ahler manifold is complex-\Cpf\ at every stage); more in general, T. Dr\v{a}ghici, T.-J. Li and W. Zhang proved in \cite[Theorem 2.3]{draghici-li-zhang} that every $4$-dimensional compact almost-complex manifold is \Cpf.

\medskip

\subsection{\Cp\ versus \Cf}
We note that, on a $4$-dimensional compact almost-complex manifold, being \Cp\ is, in fact, a consequence of being \Cf\
(in a little more general setting, on a $2n$-dimensional compact manifold one can prove that being \Cf\ at the stage $k$ implies
being \Cp\ at the stage $2n-k$, see \cite[Proposition 2.30]{li-zhang} and \cite[Theorem 2.4]{angella-tomassini}).
Here, we give two examples showing that, in higher dimension, being \Cp\ and being \Cf\ are non-related properties.

\begin{ex}{\em Being \Cf\ does not imply being \Cp.}
 Let $N_1=\Gamma_1\backslash G_1$ be a $6$-dimensional compact nilmanifold, quotient of the simply-connected nilpotent Lie group $G_1$ whose associated Lie algebra $\mathfrak{g}_1$ has structure equations (expressed with respect to a basis $\{e^i\}_{i\in\{1,\ldots,6\}}$
 of $\mathfrak{g}_1^*$, which we will confuse with the left-invariant coframe on $N_1$)
$$
\left\{
\begin{array}{rcl}
 \de e^1 &=& 0 \\[5pt]
 \de e^2 &=& 0 \\[5pt]
 \de e^3 &=& 0 \\[5pt]
 \de e^4 &=& e^1\wedge e^2 \\[5pt]
 \de e^5 &=& e^1\wedge e^4 \\[5pt]
 \de e^6 &=& e^2\wedge e^4
\end{array}
\right. \;,
$$
that is, in compact notation,
$$ \left( 0^3,\;12,\;14,\;24 \right) \;.$$
Consider the invariant complex structure on $N_1$ whose space of $(1,0)$-forms is generated, as a $\mathcal{C}^\infty(N_1)$-module, by
$$
\left\{
\begin{array}{rcl}
 \varphi^1 &:=& e^1+\im e^2 \\[5pt]
 \varphi^2 &:=& e^3+\im e^4 \\[5pt]
 \varphi^3 &:=& e^5+\im e^6
\end{array}
\right. \;;
$$
the integrability condition is easily checked, since
$$
\left\{
\begin{array}{rcl}
 2\de\varphi^1 &=& 0 \\[5pt]
 2\de\varphi^2 &=& \varphi^{1\bar1} \\[5pt]
 2\de\varphi^3 &=& -\im \varphi^{12}+\im\varphi^{1\bar2}
\end{array}
\right.
$$
(here and later, we use notations like $\varphi^{A\bar B}$ to mean $\varphi^A\wedge\bar\varphi^B$).\\
Nomizu's theorem (see \cite{nomizu}) makes the computation of the cohomology straightforward:
$$ H^2_{dR}(N_1;\C) \;=\; \C \left\langle \varphi^{13},\;\varphi^{\bar1\bar3}\right\rangle \,\oplus\, \C\left\langle \varphi^{1\bar3}-
\varphi^{3\bar1} \right\rangle \,\oplus\, \C \left\langle \varphi^{12}+\varphi^{1\bar2},\;\varphi^{2\bar1}-\varphi^{\bar1\bar2}
\right\rangle \;,$$
where harmonic representatives with respect to the metric $g=\sum_i\varphi^i\odot \bar\varphi^i$ are listed instead of their classes.
Therefore, we have
$$ H^{(2,0),(0,2)}_J\left(N_1\right) \;=\; \C \left\langle \varphi^{13},\;\varphi^{\bar1\bar3}\right\rangle \,\oplus\, \C \left\langle \varphi^{12}+\varphi^{1\bar2},\;
\varphi^{2\bar1}-\varphi^{\bar1\bar2} \right\rangle $$
and
$$ H^{(1,1)}_J\left(N_1\right) \;=\; \C\left\langle \varphi^{1\bar3}-\varphi^{3\bar1} \right\rangle \,\oplus\, \C \left\langle \varphi^{12}+\varphi^{1\bar2},\;
\varphi^{2\bar1}-\varphi^{\bar1\bar2} \right\rangle \;;$$
that is to say: \emph{$J$ is a \Cf\ complex structure which is not \Cp}.
\end{ex}

\begin{ex}{\em Being \Cp\ does not imply being \Cf.}
 Let $N_2=\Gamma_2\backslash G_2$ be a $6$-dimensional compact nilmanifold, quotient of the simply-connected nilpotent Lie group $G_2$ whose associated Lie algebra $\mathfrak{g}_2$ has structure equations
$$ \left( 0^4,\;12,\;34 \right) $$
and consider on it the complex structure given requiring that the forms
$$
\left\{
\begin{array}{rcl}
 \varphi^1 &:=& e^1+\im e^2 \\[5pt]
 \varphi^2 &:=& e^3+\im e^4 \\[5pt]
 \varphi^3 &:=& e^5+\im e^6
\end{array}
\right.
$$
are of type $(1,0)$; the integrability condition follows from
$$
\left\{
\begin{array}{rcl}
 2\de\varphi^1 &=& 0 \\[5pt]
 2\de\varphi^2 &=& 0 \\[5pt]
 2\de\varphi^3 &=& \im \varphi^{1\bar1}-\im\varphi^{2\bar2}
\end{array}
\right. \;.
$$
Nomizu's theorem (see \cite{nomizu}) gives
\begin{eqnarray*}
H^2_{dR}\left(N_2;\C\right) &=& \C \left\langle \varphi^{12},\;\varphi^{\bar1\bar2} \right\rangle \,\oplus\, \C \left\langle \varphi^{1\bar 2},\;
\varphi^{2\bar1} \right\rangle \\[5pt]
&&\oplus\,\C\left\langle \varphi^{13}+\varphi^{1\bar3},\;\varphi^{3\bar1}-\varphi^{\bar1\bar3},\;\varphi^{3\bar2}-\varphi^{\bar2\bar3},\;
\varphi^{23}-\varphi^{2\bar3} \right\rangle \;;
\end{eqnarray*}
we claim that
$$ H^{(2,0),(0,2)}_J\left(N_2\right) \;=\; \C \left\langle \varphi^{12},\;\varphi^{\bar1\bar2} \right\rangle $$
and
$$ H^{(1,1)}_J\left(N_2\right) \;=\; \C \left\langle \varphi^{1\bar 2},\;\varphi^{2\bar1} \right\rangle \;;$$
indeed, with respect to $g=\sum_i \varphi^i\odot \bar\varphi^i$, one computes
$$ \del^*\,\varphi^{13}\;=\;\del^*\,\varphi^{23}\;=\;\del^*\,\varphi^{12}\;=\;0 \;,$$
that is, $\varphi^{13}$, $\varphi^{12}$ and $\varphi^{23}$ are $g$-orthogonal to the space $\del\wedge^{1,0}$; in the same way,
one checks that
$$ \del^*\,\varphi^{1\bar2}\;=\;\delbar^*\,\varphi^{1\bar2}\;=\;\del^*\,\varphi^{1\bar3}\;=\;\delbar^*\,\varphi^{1\bar3}\;=\;0 \;;$$
that is to say: \emph{$J$ is a \Cp\ complex structure which is not \Cf}.
\end{ex}

\noindent For the sake of clearness, we collect the previous examples in the following.
\begin{prop}
 Being \Cp\ and being \Cf\ are not related properties.
\end{prop}

\medskip

\subsection{Complex-\Cpf ness for four-manifolds}
On a compact complex surface $M$ with $b^1(M)$ even, one has that the natural filtration associated to the bi-graded complex
$\wedge^{\bullet,\bullet}M$ induces a Hodge structure of weight $2$ on $H^2_{dR}(M;\,\C)$ and a Hodge structure of weight
$1$ on $H^1_{dR}(M;\,\C)$, see \cite{barth-peters-vandeven}. In particular, one has that $M$ is complex-\Cpf\ at every stage
(and so \Cpf), that is, $H^2_{dR}(M;\,\C)$ splits as
$$ H^2_{dR}(M;\,\C) \;=\; H^{(2,0)}_J(M) \,\oplus\, H^{(1,1)}_J(M) \,\oplus\, H^{(0,2)}_J(M) $$
and $H^1_{dR}(M;\,\C)$ splits as
$$ H^1_{dR}(M;\,\C) \;=\; H^{(1,0)}_J(M) \,\oplus\, H^{(0,1)}_J(M) \;. $$
(We remark that, as a consequence of the Enriques and Kodaira's classification of compact complex surfaces, actually every $4$-dimensional compact integrable almost-complex manifold $M$ with $b^1(M)$ even is a K\"ahler surface.)\\
One could ask if every $4$-dimensional compact almost-complex manifold is complex-\Cpf\ at every stage.
The following example proves that this is not true.

\begin{ex}
 Consider the standard K\"ahler structure $(J,\,\omega)$ on $\T^2_\C$ and let $\left\{J_t\right\}_t$ be the curve of almost-complex
 structures defined by
$$ J_t \;:=:\; J_{t,\,\ell} \;:=\; \left(\id - t\, L\right) \, J\, \left(\id-t\,L\right)^{-1} \;=\;
\left(
\begin{array}{cc|cc}
 && -\frac{1-t\,\ell}{1+t\,\ell} &\\
&&& -1 \\
\hline
\frac{1+t\,\ell}{1-t\,\ell} &&&\\
&1&&
\end{array}
\right)
$$
for $t$ small enough,
where
$$
L\;=\;
\left(
\begin{array}{cc|cc}
 \ell &&& \\
 & 0 &&\\
\hline
&&-\ell &\\
&&&0
\end{array}
\right) \;,
$$
with $\ell=\ell(x_2)\in\mathcal{C}^\infty(\R^4;\,\R)$ a $\Z^4$-periodic non-constant function.
For $t\neq 0$ small enough, a straightforward computation yields
$$ H^{(1,0)}_{J_t} \left(\T^2_\C\right) \;=\; \C\left\langle \de x^2+\im \de x^4 \right\rangle\;,\qquad H^{(0,1)}_{J_t}
\left(\T^2_\C\right) \;=\; \C\left\langle \de x^2-\im \de x^4 \right\rangle $$
therefore one argues that $J_t$ is not complex-\Cpf\ at the $1$-st stage, being
$$ \dim_\C H^{(1,0)}_{J_t}\left(\T^2_\C\right)+\dim_\C H^{(0,1)}_{J_t}\left(\T^2_\C\right) \;=\; 2 \;<\; 4\;=\; b^1\left(\T^2_\C\right) \;.
$$
\end{ex}

\smallskip

Moreover, one could ask if, for a $4$-dimensional almost-complex manifold, being complex-\Cpf\ at the $1$-st stage or being complex-\Cpf\ at the $2$-nd stage implies integrability: the answer is negative, as the following example shows.\\
We remark that T. Dr\v{a}ghici, T.-J. Li and W. Zhang proved in \cite[Corollary 2.14]{draghici-li-zhang} that an almost-complex structure on a $4$-dimensional compact manifold $M$ is complex-\Cpf\ at the $2$-nd stage if and only if $J$ is integrable or $h^-_J=0$.

\begin{ex}
 Consider a $4$-dimensional compact nilmanifold $M=\Gamma\backslash G$, quotient of the simply-connected nilpotent Lie group $G$ whose associated Lie algebra $\mathfrak{g}$ has structure equations
$$ \left(0^2,\;14,\;12\right) \;; $$
let $J$ be the almost-complex structure defined by
$$ J e^1 \;:=\; -e^2\; ,\qquad Je^3 \;:=\; -e^4 \;;$$
note that $J$ is not integrable, since $\textrm{Nij}(\vartheta_1,\vartheta_3)\neq 0$
(where $\left\{\vartheta_i\right\}_{i\in\{1,2,3,4\}}$ is the dual basis of $\left\{e^i\right\}_{i\in\{1,2,3,4\}}$). It has to be noted that a stronger fact holds, namely, $M$ has no integrable almost-complex structure: indeed, since $b^1(M)$ is even, if $\tilde J$ were a complex structure on $M$, then $M$ should carry a K\"ahler metric; this is not possible for compact nilmanifolds, unless they are tori.\\
One computes
$$ H^1_{dR}(M;\C)=\C\left\langle \varphi^1,\,\bar\varphi^1 \right\rangle\;,\qquad H^2_{dR}(M;\C)=\C\left\langle \varphi^{12}+
\varphi^{\bar1\bar 2},\,\varphi^{1\bar2}-\varphi^{2\bar1} \right\rangle \;.$$
Note however that $J$ is not complex-\Cpf\ at the $2$-nd stage but just \Cpf\ (the fact that $\left[\varphi^{12}+\varphi^{\bar1\bar2}\right]$ does not have a pure representative can be proved using an argument of invariance, as in Lemma \ref{lemma:invariant}). Moreover, observe that if we take the invariant almost-complex structure
$$ J' e^1 \;:=\; -e^3\; ,\qquad J'e^2 \;:=\; -e^4 \;,$$
we would have a complex-\Cpf\ at the $2$-nd stage structure which is not complex-\Cpf\ at the $1$-st stage (obviously, in this case, $h^-_{J'}=0$, according to \cite[Corollary 2.14]{draghici-li-zhang}).\\
Therefore:\\
{\itshape There exists a non-integrable almost-complex structure which is complex-\Cpf\ at the $1$-st stage.}
\end{ex}

\smallskip

\noindent In summary:
\begin{prop}\label{prop:complex-Cpf-4}
 There exist both $4$-dimensional compact almost-complex (non-complex) manifolds (with $b^1$ even) that are not complex-\Cpf\
 at the $1$-st stage and $4$-dimensional compact almost-complex (non-complex) manifolds (with $b^1$ even) that are.
\end{prop}

\section{The cone of a semi-K\"ahler manifold}\label{sec:balanced}
Let $(M,J)$ be a compact almost-complex manifold. We recall that a symplectic form $\omega$ (that is, a $\de$-closed non-degenerate $2$-form)
is said to \emph{tame} $J$ if $\omega(\sspace, J\sspace)>0$ and it is said to be \emph{compatible} with $J$ if it tames $J$ and it
is $J$-invariant, that is $\omega(J\sspace,J\sspace)=\omega(\sspace,\sspace)$. We define the \emph{tamed $J$-cone} $\mathcal{K}^t_J$ as
the set of the cohomology classes of $J$-taming symplectic forms, and the \emph{compatible $J$-cone} $\mathcal{K}^c_J$ as the set of
the cohomology classes of the $J$-compatible ones. Clearly, $\mathcal{K}^t_J$ is an open convex cone in $H^2_{dR}(M;\R)$ and $\mathcal{K}^c_J$ is a convex subcone of $\mathcal{K}^t_J$ and it is contained in $H^{(1,1)}_J(M)_\R$ (if $(M,J)$ is a compact K\"ahler manifold, then $\mathcal{K}^c_J$ is an open convex cone in $H^{1,1}_{\delbar}(M)\cap H^2_{dR}(M;\R)$); moreover, both of them are subcones of the symplectic cone in $H^2_{dR}(M;\R)$. Concerning the relation between these two cones, T.-J. Li and W. Zhang proved the following.
\begin{thm}[{\cite[Proposition 3.1, Theorem 1.1, Corollary 1.1]{li-zhang}}]\label{thm:coni-li-zhang}
Suppose that \newline $(M,J)$ is a compact almost-complex manifold with $J$ almost-K\"ahler, that is,
$\mathcal{K}^c_J\neq\varnothing$. Then we have
$$
\mathcal{K}^t_J\cap H^{(1,1)}_J(M)_\R=\mathcal{K}^c_J\;,\qquad \mathcal{K}^c_J+H^{(2,0),(0,2)}_J(M)_\R\subseteq \mathcal{K}^t_J \;.
$$
Moreover, if $J$ is \Cf, then
\begin{equation}\label{eq:thm-li-zhang}
 \mathcal{K}^t_J\;=\; \mathcal{K}^c_J+ H^{(2,0),(0,2)}_J(M)_\R \;.
\end{equation}
In particular, if $(M,J)$ is a $4$-dimensional compact almost-complex manifold with non-empty $\mathcal{K}^c_J$, then \eqref{eq:thm-li-zhang}
holds; moreover, if $b^+(M)=1$, then $\mathcal{K}^t_J=\mathcal{K}^c_J$.
\end{thm}

We want to slightly generalize this result for balanced and semi-K\"ahler manifolds. We first recall some classical definitions
and also some results from \cite{sullivan} and \cite{li-zhang}.

Firstly, let us recall (following \cite{gray-hervella}) what a balanced manifold is and let us define two more cones.

\begin{defi}
Let $(M,J)$ be a $2n$-dimensional compact almost-complex manifold. A non-degenerate $2$-form $\omega$ is said to be \emph{semi-K\"ahler} if $\omega$ is the fundamental form associated to a Hermitian metric (that is, $\omega(\sspace,J\sspace)>0$ and $\omega(J\sspace,\,J\sspace)=\omega(\sspace,\,\sspace)$) and $\de\left(\omega^{n-1}\right)=0$; the term \emph{balanced} simply refers to semi-K\"ahler forms in the integrable case.
\end{defi}

\begin{defi}
We define $\mathcal{K}b^c_J$ as the set of the cohomology classes of the $\de$-closed real $(n-1,n-1)$-forms that are $J$-compatible and positive on the complex $(n-1)$-subspaces of $T^{\C}_xM$ for each $x\in M$; similarly, we define $\mathcal{K}b^t_J$ as the set of the cohomology classes of the $\de$-closed real $(2n-2)$-forms that are positive on the complex $(n-1)$-subspaces of $T^\C_xM$ for each $x\in M$.
\end{defi}

Note that $\mathcal{K}b^c_J$ and $\mathcal{K}b^t_J$ are convex cones in $H^{2n-2}_{dR}(M;\R)$; more precisely, $\mathcal{K}b^c_J$ is a subcone of $\mathcal{K}b^t_J$ and it is contained in $H^{(n-1,n-1)}_J(M)_\R$.

The following Lemma allows us to confuse the cone $\mathcal{K}b^{c}_J$ with the cone generated by the $(n-1)$-th powers of the semi-K\"ahler forms, that is we can write
$$ \mathcal{K}b^c_J = \left\{\left[\omega^{n-1}\right] \st \omega \text{ is a semi-K\"ahler form on }(M,J)\right\} \;. $$

\begin{lemma}[{see \cite[pp. 279-280]{silva}}]
 A real $(n-1,n-1)$-form $\Phi$ positive on the complex $(n-1)$-subspaces of $T^\C_xM$ (for each $x\in M$) can be written as
$\Phi=\varphi^{n-1}$ with $\varphi$ a $J$-taming real $(1,1)$-form.
(In particular, if $\Phi$ is $\de$-closed, then $\varphi$ is semi-K\"ahler.)
\end{lemma}

Therefore, in the integrable case, i.e. for $\left(M,J\right)$ a $2n$-dimensional compact complex manifold, the cone $\mathcal{K}b^c_J$ is just the cone of balanced structures on $\left(M,J\right)$; on the other hand, in this case, $\mathcal{K}b^t_J$ is the cone of \emph{strongly Gauduchon metrics} on $\left(M,J\right)$, that is, of positive-definite $(1,1)$-forms $\gamma$ on $M$ such that the $(n,n-1)$-form $\del\left(\gamma^{n-1}\right)$ is $\delbar$-exact (see \cite[Definition 3.1]{popovici}); these metrics have been introduced by D. Popovici in \cite{popovici} as a special case of the \emph{Gauduchon metrics}, for which $\del\left(\gamma^{n-1}\right)$ is only $\delbar$-closed (see \cite{gauduchon}) (obviously, the notions of Gauduchon metric and of strongly Gauduchon metric coincide if the $\del\delbar$-Lemma holds); moreover, D. Popovici proved in \cite[Lemma 3.2]{popovici} that a $2n$-dimensional compact complex manifold $M$ carries a strongly Gauduchon metric if and only if there exists a $\de$-closed 
real $(2n-2)$-form $\Omega$ such that its component $\Omega^{(n-1,n-1)}$ of type $(n-1,n-1)$ satisfies $\Omega^{(n-1,n-1)}>0$.

\smallskip

The aim of this section is to compare $\mathcal{K}b^c_J$ and $\mathcal{K}b^t_J$
as Theorem \ref{thm:coni-li-zhang} does in the almost-K\"ahler case.

We give here a short resume of some necessary notions and results from \cite{sullivan}.\\
Recall that a \emph{cone structure} of $p$-directions on a differentiable manifold $M$ is a continuous
field $\{C(x)\}_{x\in M}$ with $C(x)$ a compact convex cone in $\wedge^p(T_xM)$. A $p$-form $\omega$
is called \emph{transverse} to a cone structure $C$ if $\omega\lfloor_x(v)>0$ for all $v\in C(x)\setminus\{0\}$,
varying $x\in M$; using the partitions of unity, it is possible to construct a transverse form for any given
$C$. Each cone structure $C$ gives rise to a cone $\mathfrak{C}$ of \emph{structure currents}, which are
the currents generated by the Dirac currents associated to the elements in $C(x)$, see \cite{sullivan};
one can prove that $\mathfrak{C}$ is a compact convex cone in $\mathcal{D}'_{p}(M)$ (we are assuming that $M$ is compact).
Moreover, one defines the cone $\mathcal{Z}\mathfrak{C}$ of the \emph{structure cycles} as the subcone
of $\mathfrak{C}$ consisting of $\de$-closed currents; we use $\mathcal{B}$ to denote the set of exact currents. Therefore, one has the cone $H\mathfrak{C}$ in $H_{p}(M;\R)$
obtained by the classes of the structure cycles. The dual cone of $H\mathfrak{C}$ is denoted by
$\breve{H}\mathfrak{C}\subseteq H^{p}_{dR}(M;\R)$ and it is characterized by the relation
$\left(\breve{H}\mathfrak{C},H\mathfrak{C}\right)\geq0$; its interior (which we will denote by
$\interno\breve{H}\mathfrak{C}$) is characterized by $\left(\interno\breve{H}\mathfrak{C},H\mathfrak{C}\right)>0$.
Finally, recall that a cone structure of $2$-directions is said to be \emph{ample} if, for
every $x\in M$, $C(x)\cap \textrm{span}\{e\in S_\tau\st \tau \text{ is a }2\text{-plane}\}\neq \{0\}$,
where $S_\tau$ is the Schubert variety given by the set of $2$-planes intersecting $\tau$ in at least one line;
by \cite[Theorem III.2]{sullivan}, an ample cone structure admits non-trivial structure cycles.

We are interested in the following cone structures. Given a $2n$-dimensional compact almost-complex manifold $(M,J)$ and
fixed $p\in\{0,\ldots, n\}$, let $C_{p,J}$ be the cone whose elements are the compact convex cones $C_{p,J}(x)$
generated by the positive combinations of $p$-dimensional complex subspaces of $T^\C_xM$ belonging to $\wedge^{2p}(T^\C_xM)$. We will
denote with $\mathfrak{C}_{p,J}$ the compact convex cone (see \cite[III.4]{sullivan}) of the structure
currents (also called \emph{complex currents}) and with $\mathcal{Z}\mathfrak{C}_{p,J}$ the compact convex
cone (see \cite[III.7]{sullivan}) of the structure cycles (also called \emph{complex cycles}). The structure
cone $C_{1,J}$ is ample, see \cite[p. 249]{sullivan}, therefore it admits non-trivial cycles.

We need also the following (see e.g. \cite[Proposition I.1.3]{silva}).

\begin{lemma}
 A structure current in $\mathfrak{C}_{p,J}$ is a positive current of bi-dimension $(p,p)$.
\end{lemma}

Note that $\mathcal{K}b^t_J$ can be identified with the set of the classes of $\de$-closed $(2n-2)$-forms transverse to $C_{n-1,J}$. As
an application of a general fact, we get the following.

\begin{thm}[see {\cite[Theorem I.7]{sullivan}}]\label{thm:Kbt}
 Let $(M,J)$ be a $2n$-dimensional compact almost-complex manifold. Then $\mathcal{K}b^t_J$ is non-empty if and only if there
is no non-trivial $\de$-closed positive currents of bi-dimension $(n-1,n-1)$ that is a boundary, i.e.
$$
\mathcal{Z}\mathfrak{C}_{n-1,J}\cap\mathcal{B}=\{0\}\,.
$$
Furthermore, suppose that $0\not\in\mathcal{K}b^t_J$: then
 $\mathcal{K}b^t_J\subseteq H^{2n-2}_{dR}(M;\R)$ is the interior of the dual cone in $H^{2n-2}_{dR}(M;\R)$ of
$H\mathfrak{C}_{n-1,J}$.
\end{thm}

\begin{proof}
 For the ``only if'' part, note that if $\omega\in\mathcal{K}b^t_J\neq\varnothing$ and $\eta$ is
a non-trivial $\de$-closed positive current of bi-dimension $(n-1,n-1)$ and a boundary (let $\eta=:\de\xi$), then
$$ 0 < \left(\eta,\pi^{n-1,n-1}\omega\right) = \left(\eta,\omega\right) = \left(\de\xi,\omega\right)
=\left(\xi,\de\omega\right)=0 $$
yields an absurd.\\
For the ``if part'', if no non-trivial $\de$-closed positive currents of bi-dimension $(n-1,n-1)$ that is a
boundary exists then, by \cite[Theorem I.7(ii)]{sullivan}, there exists a $\de$-closed form that is transverse to $C_{n-1,J}$,
that is, $\mathcal{K}b^t_J$ is non-empty.\\
Also the last statement follows from \cite[Theorem I.7(iv)]{sullivan}. Indeed, the
assumption $0\not\in\mathcal{K}b^t_J$ means that no $\de$-closed transverse form is cohomologous to zero therefore, by
\cite[Theorem I.7(iii)]{sullivan}, it assures the existence of non-trivial structure cycles.
\end{proof}

We give a similar characterization for $\mathcal{K}b^c_J$, see \cite{li-zhang}.
\begin{thm}\label{thm:Kbc}
 Let $(M,J)$ be a $2n$-dimensional compact almost-complex manifold. Suppose that $\mathcal{K}b^c_J\neq\varnothing$ and that $0\not\in\mathcal{K}b^c_J$. Then
 $\mathcal{K}b^c_J\subseteq H^{(n-1,n-1)}_J(M)_\R$ is the interior of the dual cone in $H_J^{(n-1,n-1)}(M)_\R$ of $H\mathfrak{C}_{n-1,J}$.
\end{thm}

\begin{proof}
 By the hypothesis $0\not\in\mathcal{K}^c_J$, we have that $\left(\mathcal{K}b^c_J,H\mathfrak{C}_{n-1,J}\right)>0$ and therefore the inclusion
 $\mathcal{K}b^c_J\subseteq\interno\breve{H}\mathfrak{C}_{n-1,J}$ holds.\\
 To prove the other inclusion, let $e\in H^{(n-1,n-1)}_J(M)_\R$ be an element in the interior of the dual cone in $H_J^{(n-1,n-1)}$ of $H\mathfrak{C}_{n-1,J}$, i.e. $e$ is such that $\left(e,H\mathfrak{C}_{n-1,J}\right)>0$.
 Consider the isomorphism
$$
\overline{\sigma}^{n-1,n-1}\colon
H^{(n-1,n-1)}_J(M)_\R \stackrel{\simeq}{\longrightarrow} \left(\frac{\overline{\pi_{n-1,n-1}\mathcal{Z}}}{\overline{\pi_{n-1,n-1}\mathcal{B}}}
\right)^{*}
$$
(see \cite[Proposition 2.4]{li-zhang}): $\overline{\sigma}^{n-1,n-1}(e)$ gives rise to a functional on $\frac{\overline{\pi_{n-1,n-1}\mathcal{Z}}}{\overline{\pi_{n-1,n-1}\mathcal{B}}}$, that is, to a functional on $\overline{\pi_{n-1,n-1}\mathcal{Z}}$ which vanishes on $\overline{\pi_{n-1,n-1}\mathcal{B}}$; such a functional, in turn, gives rise to a hyperplane $L$ in $\overline{\pi_{n-1,n-1}\mathcal{Z}}$ containing $\overline{\pi_{n-1,n-1}\mathcal{B}}$, which is closed in $\mathcal{D}'_{n-1,n-1}(M)_\R$ (since a kernel hyperplane in a closed set) and disjoint from $\mathfrak{C}_{n-1,J}\setminus\{0\}$ by the choice made for $e$.\newline
Pick a Hermitian metric and let $\varphi$ be its fundamental form; set
$$
K:=\{T\in \mathfrak{C}_{n-1,J}\st T(\varphi^{n-1})=1\}\,;
$$
then $K$ is a compact set.\newline
Now, in the space $\mathcal{D}'_{n-1,n-1}(M)_\R$ of
real
$(n-1,n-1)$-currents, consider the two sets $L$ (which is a closed subspace) and $K$ (which is a compact convex
non-empty set), which have empty intersection. By
the Hahn-Banach separation theorem, there exists a hyperplane containing $L$ (and then also
$\overline{\pi_{n-1,n-1}\mathcal{B}}$)
and disjoint from $K$. The functional on $\mathcal{D}'_{n-1,n-1}(M)_\R$ associated to this
hyperplane is a real $(n-1,n-1)$-form which
is $\de$-closed (since it vanishes on $\overline{\pi_{n-1,n-1}\mathcal{B}}$) and positive on the complex $(n-1)$-subspaces of $T^\C_xM$.
\end{proof}

A similar proof yields the following (see \cite{harvey-lawson}, \cite{li-zhang}).

\begin{thm}
 Let $(M,J)$ be a $2n$-dimensional compact almost-complex manifold.
\begin{enumerate}
 \item[(i)] Assuming that $J$ is integrable, there exists a K\"ahler metric if and only if
 $\mathfrak{C}_{1,J}\cap\pi_{1,1}\mathcal{B}=\{0\}$.
 \item[(ii)] There exists an almost-K\"ahler metric if and only if $\mathfrak{C}_{1,J}\cap \overline{\pi_{1,1}\mathcal{B}}=\{0\}$.
 \item[(iii)] There exists a semi-K\"ahler metric if and only if $\mathfrak{C}_{n-1,J}\cap\overline{\pi_{n-1,n-1}\mathcal{B}}=\{0\}$.
\end{enumerate}
\end{thm}

\begin{proof}
 Note that {\itshape(i)} is a consequence of {\itshape(ii)}, since, if $J$ is integrable, then $J$ is closed
 (see \cite[Lemma 6]{harvey-lawson}), that is, $\pi_{1,1}\mathcal{B}$ is a closed set. We prove {\itshape(iii)},
 following closely the proof of {\itshape(i)} in \cite[Proposition 1.2, Theorem 1.4]{harvey-lawson}.\\
For the ``only if'' part, observe that if $\omega$ is a semi-K\"ahler form and
$$
0\neq \eta\in\mathfrak{C}_{n-1,J}\cap\overline{\pi_{n-1,n-1}\mathcal{B}}\neq\{0\}\,,
$$
suppose $\eta=\lim_{k\to+\infty} \pi_{n-1,n-1}\left(\de\alpha_k\right)$ for $\left(\alpha_k\right)_{k\in\N}\subseteq \mathcal{D}'_{2n-1}(M)$,
then we get an absurd, since
\begin{eqnarray*}
 0 &<& \left(\eta,\omega^{n-1}\right)=\left(\lim_{k\to+\infty}\pi_{n-1,n-1}\left(\de\alpha_k\right),\omega^{n-1}\right)\\[5pt]
&=&\lim_{k\to+\infty}\left(\de\alpha_k,\omega^{n-1}\right)=\lim_{k\to+\infty}\left(\alpha_k,\de\omega^{n-1}\right)=0 \;.
\end{eqnarray*}
For the ``if part'', fix again a Hermitian metric (we will call
$\varphi$ its fundamental form) and set $K:=\left\{T\in\mathfrak{C}_{n-1,J}\st T\left(\varphi^{n-1}\right)=1\right\}$; $K$ is a compact convex non-empty
set in $\mathcal{D}'_{n-1,n-1}(M)_\R$. By the Hahn-Banach separation theorem, we can find a hyperplane in
$\mathcal{D}'_{n-1,n-1}(M)_\R$ containing the closed subspace $\overline{\pi_{n-1,n-1}\mathcal{B}}$ and disjoint
from $K$; the real $(n-1,n-1)$-form associated to this hyperplane is then a $\de$-closed real $(n-1,n-1)$-form positive on the complex $(n-1)$-subspaces, that is, it gives a semi-K\"ahler form.
\end{proof}

Finally, we can prove a theorem similar to \cite[Proposition 3.1]{li-zhang} (which we quote in Theorem \ref{thm:coni-li-zhang}).

\begin{thm}\label{thm:coni-semiK}
 Let $(M,J)$ be a $2n$-dimensional compact almost-complex manifold. Assume
 that $\mathcal{K}b^c_J\neq\varnothing$ (that is, $(M,J)$
 is semi-K\"ahler) and that $0\not\in\mathcal{K}b^t_J$. Then
\begin{equation}\label{eq:cono-1}
\mathcal{K}b^t_J\cap H^{(n-1,n-1)}_J(M)_\R = \mathcal{K}b^c_J
\end{equation}
and
\begin{equation}\label{eq:cono-2}
\mathcal{K}b^c_J+H^{(n,n-2),(n-2,n)}_J(M)_\R \subseteq \mathcal{K}b^t_J \;.
\end{equation}
Moreover, if $J$ is \Cf\ at the $(2n-2)$-th stage, then
\begin{equation}\label{eq:cono-3}
\mathcal{K}b^c_J+H^{(n,n-2),(n-2,n)}_J(M)_\R = \mathcal{K}b^t_J \;.
\end{equation}
\end{thm}

\begin{proof}
Since $\mathcal{K}b^t_J$ is the interior of the dual cone in $H^{2n-2}_{dR}(M;\R)$ of $H\mathfrak{C}_{n-1,J}$ by Theorem \ref{thm:Kbt} and $\mathcal{K}b^c_J$ is the interior of the dual cone in $H_{(n-1,n-1)}^J(M)_\R$ of $H\mathfrak{C}_{n-1,J}$ by Theorem \ref{thm:Kbc}, the formula \eqref{eq:cono-1} is proved.\\
The inclusion \eqref{eq:cono-2} is easily proved, since the sum of a semi-K\"ahler form and a $J$-anti-invariant $(2n-2)$-form is again $\de$-closed and positive on the complex $(n-1)$-subspaces.\\
Finally, \eqref{eq:cono-3} follows from
\begin{eqnarray*}
\mathcal{K}b^t_J &=& \interno\breve{H}\mathfrak{C}_{n-1,J}=\interno\breve{H}\mathfrak{C}_{n-1,J} \cap H^2_{dR}(M;\R)\\[5pt]
&=&\interno\breve{H}\mathfrak{C}_{n-1,J} \cap \left(H^{(n-1,n-1)}_J(M)_\R+H^{(n,n-2),(n-2,n)}_J(M)_\R\right)\\[5pt]
&\subseteq& \mathcal{K}b^c_J + H^{(n,n-2),(n-2,n)}_J(M)_\R
\end{eqnarray*}
and from \eqref{eq:cono-2}.
\end{proof}

\section{A cohomological property for semi-K\"ahler manifolds}\label{sec:Cpf-and-pf}
In \cite[Theorem 4.1]{fino-tomassini}, it is proved that a $2n$-dimensional compact \Cpf\ almost-K\"ahler manifold
whose symplectic form $\omega$ satisfies the \emph{Hard Lefschetz Condition} (that is, $\omega^k:H^{n-k}_{dR}(M;\R)\to H^{n+k}_{dR}(M;\R)$
is an isomorphism for every $k\in\{0,\ldots,n\}$) is \pf, too (note that by \cite[Proposition 3.2]{fino-tomassini}, see
also \cite[Proposition 2.8]{draghici-li-zhang}, every compact almost-K\"ahler manifold is \Cp).

The previous result can be generalized a little. In order to look into the cohomology of balanced manifolds
and the duality between $H^{(\bullet,\bullet)}_J(M)$ and $H^J_{(\bullet,\bullet)}(M)$, we get the following.
\begin{prop}\label{prop:balanced-pf}
 Let $(M,J)$ be a $2n$-dimensional compact almost-complex manifold endowed with a semi-K\"ahler form $\omega$. Suppose that $\omega^{n-1}: H^1_{dR}(M;\R) \to H^{2n-1}_{dR}(M;\R)$ is an isomorphism. If $J$ is complex-\Cpf\ at the $1$-st stage, then it is also complex-\pf\ at the $1$-st stage and
$H^{(1,0)}_J(M)\simeq H_{(0,1)}^J(M)$.
\end{prop}

\begin{proof}
 Firstly, note that $J$ is complex-\p\ at the $1$-st stage: indeed, if
$$
\mathfrak{a}\in H^J_{(1,0)}(M)\cap H^J_{(0,1)}(M)\,,
$$
then $\mathfrak{a}\lfloor_{H^{(1,0)}_J(M)}=0
=\mathfrak{a}\lfloor_{H^{(0,1)}_J(M)}$ and then $\mathfrak{a}=0$, since, by hypothesis,
$$
H^{1}_{dR}(M;\C)
=H^{(1,0)}_J(M)\oplus H^{(0,1)}_J(M)\,.
$$
Now, note that the map
$$
H^{(1,0)}_J(M) \stackrel{\omega^{n-1}\wedge\sspace}{\longrightarrow} H^{(n,n-1)}_J(M)
\stackrel{\varphi\mapsto\int_M\varphi\wedge\sspace}{\longrightarrow} H_{(0,1)}^J(M)
$$
induces an injection in cohomology $H^{(1,0)}_J(M)\hookrightarrow H^J_{(0,1)}(M)$; similarly, one has $H^{(0,1)}_J(M)\hookrightarrow H^J_{(1,0)}(M)$; since
$$
H^{(1,0)}_J(M)\oplus H^{(0,1)}_J(M)=H^1_{dR}(M;\C)\simeq H_1(M;\C)\,,
$$
we have the thesis.
\end{proof}

\begin{ex}
 The first example of a $6$-dimensional balanced manifold whose $\omega^2$ is an isomorphism in cohomology is the
\emph{Iwasawa manifold} $\mathbb{I}_3$, whose complex structure equations are
$$
\de\varphi^1=\de\varphi^2=0\,, \quad\de\varphi^3=-
\varphi^1\wedge\varphi^2\,,
$$
and endowed with the balanced form
$$
\omega:=\frac{\im}{2}(\varphi^1\wedge\bar\varphi^1+
\varphi^2\wedge\bar\varphi^2+\varphi^3\wedge\bar\varphi^3);
$$
its cohomology is
$$
H^1_{dR}(\mathbb{I}_3;\C)=\C\left\langle \varphi^1,\,\varphi^2,\,\bar\varphi^1,\,\bar\varphi^2 \right\rangle
$$
and
$$
H^5_{dR}(\mathbb{I}_3;\C)=\C\left\langle \varphi^{123\bar1\bar3},\,\varphi^{123\bar2\bar3},\,\varphi^{13\bar1\bar2\bar3},
\,\varphi^{23\bar1\bar2\bar3} \right\rangle\,,
$$
so $\omega^2\colon H^1_{dR}(\mathbb{I}_3;\C)\to H^5_{dR}(\mathbb{I}_3;\C)$ is an isomorphism. Therefore, $\mathbb{I}_3$ is complex-\Cpf\ at the $1$-st stage and complex-\pf\ at the $1$-st stage (note that, being $\varphi^1$, $\varphi^2$, $\bar\varphi^1$, $\bar\varphi^2$ the harmonic representatives, with respect to the metric $\sum_{i=1}^3\varphi^i\odot\bar\varphi^i$, of $H^1_{dR}\left(\mathbb{I}_3;\C\right)$, the result follows also from \cite[Theorem 3.7]{fino-tomassini}).
\end{ex}

 We also remark that the hypothesis that $\omega^{n-1}$ is an isomorphism is not
trivially satisfied. The following example gives a $6$-dimensional almost-complex manifold endowed
with a non-degenerate $J$-compatible $2$-form $\omega$,
such that $\de\omega^2=0$ but $\omega^2:H^1_{dR}(M;\R)\to H^5_{dR}(M;\R)$ is not an isomorphism.

\begin{ex}\label{ex:fino-tomassini}
Consider a nilmanifold $M=\Gamma\backslash G$, quotient of the simply-connected nilpotent Lie group whose associated Lie algebra $\mathfrak{g}$ has structure equations
$$ \left(0^4,\,12,\,13\right)\;. $$
In \cite{fino-tomassini}, the almost-complex structure $J'e^1:=-e^2,\,J'e^3:=-e^4,\,J'e^5:=-e^6$ is
considered as a first example of a compact non-\Cp\ manifold; note that it is not even \Cf: indeed, for example, one can prove that the cohomology class $\left[e^{15}+e^{16}\right]$ has neither $J'$-invariant nor $J'$-anti-invariant representatives, using an argument of invariance as in Lemma \ref{lemma:invariant}. Here, we consider the almost-complex structure
$$
J e^1 \;:=\; -e^5\;, \qquad J e^2 \;:=\; -e^3\;, \qquad Je^4 \;:=\; -e^6
$$
and the non-degenerate $J$-compatible $2$-form
$$
\omega:= e^{15}+e^{23}+e^{46}\,.
$$
A straightforward computation gives $\de\omega=-e^{134}\neq0$ and
$$
\de\omega^2=\de\left(e^{1235}-e^{1456}+e^{2346}\right)=0\,.
$$
We have that
$$
H^{1}_{dR}(M;\R)=\R\left\langle e^1,e^2,e^3,e^4 \right\rangle
$$
and
$$
\omega^2e^1=e^{12346}=\de e^{3456} \;.
$$
Therefore, $\omega^2\colon H^{1}_{dR}(M;\R)\to H^{5}_{dR}(M;\R)$ is not injective.
\end{ex}

We give two explicit examples of $2n$-dimensional balanced manifolds (with $2n=6$ and $2n=10$ respectively) whose $\omega^{n-1}$ is an isomorphism and with small balanced deformations.

\begin{ex}\label{ex:eta-beta-bilanciata}
Let $\eta\beta_5$ be the $10$-dimensional nilmanifold considered in \cite{alessandrini-bassanelli}
to prove that being $p$-K\"ahler is not a stable property under small
deformations of the complex structure. We recall that the complex structure equations are
$$
\left\{
\begin{array}{rcl}
 \de\varphi^1 &=& 0 \\[5pt]
 \de\varphi^2 &=& 0 \\[5pt]
 \de\varphi^3 &=& 0 \\[5pt]
 \de\varphi^4 &=& 0 \\[5pt]
 \de\varphi^5 &=& -\varphi^{12}-\varphi^{34}
\end{array}
\right. \;.
$$
The manifold $\eta\beta_5$ admits the balanced form
$$ \omega:=\frac{\im}{2}\sum_{h=1}^5\varphi^h\wedge\bar\varphi^h \;.$$
By Hattori's theorem (see \cite{hattori}), one computes
$$ H^1_{dR}(\eta\beta_5;\C)=\C\left\langle \varphi^1,\,\varphi^2,\,\varphi^3,\,\varphi^4,\,\bar\varphi^1,\,\bar\varphi^2,\,\bar\varphi^3,\,\bar\varphi^4\right\rangle $$
and
\begin{eqnarray*}
H^9_{dR}(\eta\beta_5;\C)&=&\C\left\langle \varphi^{12345\bar2\bar3\bar4\bar5},\,\varphi^{12345\bar1\bar3\bar4\bar5},\,\varphi^{12345\bar1\bar2\bar4\bar5},\,\varphi^{12345\bar1\bar2\bar3\bar5},\right.\\[5pt]
&&\left.\varphi^{2345\bar1\bar2\bar3\bar4\bar5},\,\varphi^{1345\bar1\bar2\bar3\bar4\bar5},\,\varphi^{1245\bar1\bar2\bar3\bar4\bar5},\,\varphi^{1235\bar1\bar2\bar3\bar4\bar5}\right\rangle \;;
\end{eqnarray*}
therefore, $\eta\beta_5$ is complex-\Cpf\ at the $1$-st stage with $\omega^4\colon H^1_{dR}\left(\eta\beta_5;\C\right)\to H^9_{dR}\left(\eta\beta_5;\C\right)$ an isomorphism and so it is also complex-\pf\ at the $1$-st stage by Proposition \ref{prop:balanced-pf} (note that, being the listed representatives harmonic with respect to the metric $\sum_{h=1}^5\varphi^h\odot\bar\varphi^h$, the same result follows from an argument similar to \cite[Theorem 3.7]{fino-tomassini}).
\smallskip

\noindent Now, consider the complex deformations defined by
$$
\left\{
\begin{array}{rcl}
 \varphi^1_t &:=& \varphi^1+t\,\bar\varphi^1 \\[5pt]
 \varphi^2_t &:=& \varphi^2 \\[5pt]
 \varphi^3_t &:=& \varphi^3 \\[5pt]
 \varphi^4_t &:=& \varphi^4 \\[5pt]
 \varphi^5_t &:=& \varphi^5
\end{array}
\right. \;,
$$
for $t\in\C$ small enough;
it is straightforward to compute
$$
\left\{
\begin{array}{rcl}
 \varphi^1 &=& \frac{1}{1-\left|t\right|^2}\,\left(\varphi^1_t-t\,\bar\varphi^1_t\right) \\[5pt]
 \varphi^2 &=& \varphi^2_t \\[5pt]
 \varphi^3 &=& \varphi^3_t \\[5pt]
 \varphi^4 &=& \varphi^4_t \\[5pt]
 \varphi^5 &=& \varphi^5_t
\end{array}
\right. \;;
$$
therefore, if $t$ is small enough, then the complex structure equations are
$$
\left\{
\begin{array}{rcl}
 \de\varphi^1_t &=& 0 \\[5pt]
 \de\varphi^2_t &=& 0 \\[5pt]
 \de\varphi^3_t &=& 0 \\[5pt]
 \de\varphi^4_t &=& 0 \\[5pt]
 \de\varphi^5_t &=& -\frac{1}{1-\left|t\right|^2}\,\varphi^{12}_t-\varphi^{34}_t-\frac{t}{1-\left|t\right|^2}\,
\varphi^{2\bar1}_t
\end{array}
\right. \;.
$$
The complex structures $J_t$ on $\eta\beta_5$ are balanced.\newline
Indeed, $\omega_t:=\frac{\im}{2}\displaystyle\sum_{h=1}^5\varphi^h_t\wedge\overline{\varphi}^h_t$ gives rise to a curve of
balanced structures on $\eta\beta_5$. Furthermore, $J_t$ are complex-\Cpf\ at the $1$-st stage and
$$
\omega_t^{4}: H^1_{dR}(\eta\beta_5;\R) \to H^{9}_{dR}(\eta\beta_5;\R)
$$
are isomorphisms. Therefore,
according to Proposition \ref{prop:balanced-pf}, it follows that the complex structures $J_t$ are
complex-\pf\ at the $1$-st stage and $H^{(1,0)}_{J_t}(M)\simeq H_{(0,1)}^{J_t}(M)$.
\end{ex}

\begin{ex}
 Consider a solvmanifold $M=\Gamma\backslash G$, quotient of the simply-connected completely-solvable Lie group $G$ whose associated Lie algebra $\mathfrak{g}$ has structure equations
$$ \left(0,\,-12,\,34,\,0,\,15,\,46\right) \;,$$
endowed with the almost-complex structure
$$
\left\{
\begin{array}{rcl}
 \varphi^1 &:=& e^1+\im e^4 \\[5pt]
 \varphi^2 &:=& e^2+\im e^5 \\[5pt]
 \varphi^3 &:=& e^3+\im e^6
\end{array}
\right.
$$
and the $J$-compatible symplectic structure
$$ \omega := e^{14}+e^{25}+e^{36} $$
(see also \cite[\S6.3]{fino-tomassini}).\\
The complex structure equations are
$$
\left\{
\begin{array}{rcl}
 \de\varphi^1 &=& 0 \\[5pt]
 2\,\de\varphi^2 &=& -\varphi^{1\bar2}-\varphi^{\bar1\bar2} \\[5pt]
 2\im \,\de\varphi^3 &=& -\varphi^{1\bar3}+\varphi^{\bar1\bar3}
\end{array}
\right. \;.
$$
By Hattori's theorem (see \cite{hattori}), it is straightforward to compute
$$ H^1_{dR}(M;\R)=\R\left\langle e^1,\,e^4\right\rangle\;,\qquad H^5_{dR}(M;\R)=\R\left\langle *e^1,\,*e^4 \right\rangle = \R\left\langle e^{23456},\, e^{12356} \right\rangle\;.$$
Now, consider the small almost-complex deformations of $M$ defined by
$$
\left\{
\begin{array}{rcl}
 \varphi^1_t &:=& \varphi^1 \\[5pt]
 \varphi^2_t &:=& \varphi^2+\im\,t\,e^6 \\[5pt]
 \varphi^3_t &:=& \varphi^3
\end{array}
\right.
$$
with the non-degenerate $J$-compatible $2$-form
$$ \omega_t := e^{14}+e^{25}+e^{36}+t\,e^{26} \;;$$
for $t\neq0$, one has that $\de\omega\neq 0$, but
$$ \omega_t^2=\omega_0^2-t\,e^{1246} $$
so
$$ \de\omega^2_t=0 \;.$$
Moreover:
$$ \omega_t^2\, e^1\;=\;e^{12356}\;,\qquad \omega_t^2\, e^4\;=\;e^{23456}\;.$$
\end{ex}

\section{The semi-continuity problem}\label{sec:semicontinuity}
T. Dr\v{a}ghici, T.-J. Li, W. Zhang proved in \cite{draghici-li-zhang1} that, given a family $\{J_t\}_t$ of (\Cpf)
almost-complex structures on a $4$-dimensional compact manifold $M$, the dimension
$$ h^+_{J_t}(M) \;:=\; \dim_\R H^{+}_{J_t}(M) $$
is a lower-semi-continuous function in $t$, and therefore the dimension
$$ h^-_{J_t}(M) \;:=\; \dim_\R H^{-}_{J_t}(M) $$
is an upper-semi-continuous function in $t$.

Their result is closely related to the geometry of four-manifolds. We are interested in establish
if a similar result could occur in dimension higher than $4$ assuming further hypotheses. But, first of all,
we provide several examples showing that things are more complicated in dimension higher than $4$.

In Example \ref{ex:etabeta} we construct a curve $\left\{J_t\right\}_t$ of almost-complex structures such that $h^-_{J_t}$ is not
upper-semi-continuous. Therefore, we have the following

\begin{prop}
 There exists a compact \Cpf\ complex manifold $(\eta\beta_5,\, J)$ of real dimension $10$ and there exists a curve
$\left\{J_t\right\}_t$ of almost-complex structures (which are non-\Cp\ for $t\neq 0$) with $J_0=J$ such that $h^-_{J_t}\left(\eta\beta_5\right)$
is not an upper-semi-continuous function in $t$.
\end{prop}

\begin{ex}\label{ex:etabeta}
Consider again the nilmanifold $\eta\beta_5$, see Example \ref{ex:eta-beta-bilanciata}.\\
By Nomizu's theorem (see \cite{nomizu}), it is straightforward to compute
$$
\begin{array}{lll}
 H^2_{dR}(\eta\beta_5;\,\C) &=& \C\left\langle \varphi^{13},\,
\varphi^{14},\,\varphi^{23},\,\varphi^{24},\, \varphi^{\bar1\bar3},\,\varphi^{\bar1\bar4},\,
\varphi^{\bar2\bar3},\,\varphi^{\bar2\bar4},\, \varphi^{12-34},\,\varphi^{\bar1\bar2-\bar3\bar4}
\right\rangle\\[7pt]
&&\oplus\; \C\left\langle \varphi^{1\bar1},\,\varphi^{1\bar2},\,\varphi^{1\bar3},\,\varphi^{1\bar4},\,
\varphi^{2\bar1},\,\varphi^{2\bar2},\,\varphi^{2\bar3},\,\varphi^{2\bar4},\right.\\[5pt]
 && \left.\varphi^{3\bar1},\,\varphi^{3\bar2},\,\varphi^{3\bar3},\,\varphi^{3\bar4},\, \varphi^{4\bar1},\,\varphi^{4\bar2},\,\varphi^{4\bar3},\,\varphi^{4\bar4}
\right\rangle \;,
\end{array}
$$
so one notes immediately that $\eta\beta_5$ is a \Cpf\ complex manifold with
$$ h^-_J\left(\eta\beta_5\right) \;=\; 10\;,\qquad\qquad h^+_J\left(\eta\beta_5\right)\;=\;16\;.$$
Now, consider the curve of complex structures already considered in Example \ref{ex:eta-beta-bilanciata}.\\
For $t\neq 0$ small enough, the complex structures $J_t$ are not \Cp: indeed, one has
$$ H^+_{J_t}(\eta\beta_5) \;\ni\;
\left[\frac{t}{1-\left|t\right|^2}\,\varphi^{2\bar1}_t+\de\varphi^5_t\right] \;=\; \left[-\frac{1}{1-\left|t\right|^2}\,
\varphi^{12}_t-\varphi^{34}_t\right]\;\in\; H^-_{J_t} \;, $$
where $\left[\frac{t}{1-\left|t\right|^2}\,\varphi^{2\bar1}_t\right]$ is a non-zero cohomology class (again by Nomizu's theorem).
Moreover, one has
$$
H^-_{J_t}(\eta\beta_5) \;\supseteq \; \C\left\langle
\varphi^{13}_t,\,\varphi^{14}_t,\,\varphi^{23}_t,\,\varphi^{24}_t,\, \varphi^{\bar1\bar3}_t,\,\varphi^{\bar1\bar4}_t,
\,\varphi^{\bar2\bar3}_t,\,\varphi^{\bar2\bar4}_t,\, \varphi^{12}_t,\, \varphi^{34}_t,\, \varphi^{\bar1\bar2}_t,\,
\varphi^{\bar3\bar4}_t
\right\rangle \;,
$$
so (for $t\neq0$ small enough)
$$
 h^-_{J_0} \;=\; 10 \;<\; 12 \;\leq\; h^-_{J_t}
$$
hence $h^-_{J_t}$ is not an upper-semi-continuous functions in $t=0$.
\end{ex}

We provide also Example \ref{ex:s3t3}, which shows that the lower-semi-continuity of $h^+_{J_t}$ can fail. Therefore, we have the following result.

\begin{prop}
 Consider the $6$-dimensional compact manifold $\mathbb{S}^3\times\T^3$. It admits a \Cf\ complex structure $J$ and a curve
$\left\{J_t\right\}_t$ of almost-complex structures (which are not \Cp) such that $J_0=J$ and $h^+_{J_t}$
is not a lower-semi-continuous function in $t$.
\end{prop}

\begin{ex}\label{ex:s3t3}
 Consider $\mathbb{S}^3\times\T^3$, whose structure equations (with respect to the global coframe $\left\{e^j\right\}_{j\in\{1,\ldots,6\}}$) are
$$ \left( 23, \, -13, \, 12, \, 0^3 \right) $$
and with the almost-complex structure $J$ defined by
$$
\left\{
\begin{array}{rcl}
 \varphi^1 &:=& e^1+\im e^4 \\[5pt]
 \varphi^2 &:=& e^2+\im e^5 \\[5pt]
 \varphi^3 &:=& e^3+\im e^6
\end{array}
\right. \;.
$$
We have that
\begin{eqnarray*}
H^2_{dR}(\mathbb{S}^3\times\T^3;\C)&=&\C\left\langle \left[e^{45}\right],\, \left[e^{46}\right],\, \left[e^{56}\right] \right\rangle \\[5pt]
&=&  \left\langle \left[\varphi^{12}+\varphi^{\bar1\bar2}\right],\,\left[\varphi^{13}+\varphi^{\bar1\bar3}\right],\, \left[\varphi^{23}+\varphi^{\bar2\bar3}\right] \right\rangle =H^{-}_{J}\\[5pt]
&=&  \left\langle \left[\varphi^{1\bar2}-\varphi^{2\bar1}\right],\,\left[\varphi^{1\bar3}-\varphi^{3\bar1}\right],\, \left[\varphi^{2\bar3}-\varphi^{3\bar2}\right] \right\rangle =H^{+}_{J} \;.
\end{eqnarray*}
Consider the following small deformations of $J$:
$$
\left\{
\begin{array}{rcl}
 \varphi^1_t &:=& \varphi^1 + t \,\bar\varphi^1 \\[5pt]
 \varphi^2_t &:=& \varphi^2 \\[5pt]
 \varphi^3_t &:=& \varphi^3
\end{array}
\right.
$$
for $t\in\C$ small enough.\\
We have that
$$ \varphi^{1\bar2}-\varphi^{2\bar1}=\frac{1}{1-|t|^2}\left(\varphi_t^{1\bar2}-\varphi_t^{2\bar1}\right)+
\frac{1}{1-|t|^2}\left(\bar t\,\varphi_t^{12}-t\,\varphi_t^{\bar1\bar2}\right) \;:$$
we claim that the term $\psi:=\bar t\,\varphi_t^{12}-t\,\varphi_t^{\bar1\bar2}$ can not be
written as the sum of a $J_t$-invariant form and an exact form. Indeed, being $\psi$ invariant, taking
the average, we can suppose that the $J_t$-anti-invariant component of the exact term is the differential of an invariant
form (see Lemma \ref{lemma:invariant} for a similar argument). It is easy to check that this is not possible if $t\not\in\R$. Therefore (since the same argument can be repeated for $\varphi^{1\bar3}+\varphi^{\bar13}$) we have that, for $t\not\in\R$
small enough,
$$ h^+_{J_t}=1 <3 = h^+_{J_0} \;,$$
that is, $h^+_{J_t}$ is not a lower-semi-continuous function in $t=0$.
\end{ex}

These examples force us to consider stronger conditions under which semi-continuity could maybe occur.

\medskip

Now we turn to the aim to give a more precise statement of the problem. We remark that, in dimension $4$,
we does not have only the semi-continuity property, but we have also that each $J$-invariant class has a
$J_t$-invariant class {\em close to it}. This is also a sufficient condition to assure that, if $\alpha$
is a $J$-compatible symplectic structure, there is a $J_t$-compatible symplectic structure $\alpha_t$ for
$t$ small enough. Therefore, we are interested in the following problem:\\[6pt]
\noindent{\itshape
 Let $(M,J)$ be a compact almost-complex manifold with
$$
H^+_J \;=\; \C\left\langle \left[\alpha^1\right],\,\ldots,\,\left[\alpha^k\right]\right\rangle \;,
$$
where $\alpha^1,\,\ldots,\,\alpha^k$ are forms of type $(1,1)$ with respect to $J$. We look for further
hypotheses assuring that all the almost-complex structures of the curve $\left\{J_t\right\}_t$ (for $t$ small enough) have
$$ H^+_{J_t} \;\supseteq\; \C\left\langle \left[\alpha^1_t\right],\,\ldots,\,\left[\alpha^k_t\right]\right\rangle $$
with $\alpha^j_t=\alpha^j+\opiccolo{1}$. In particular, $h^+_{J_t}$ is a lower-semi-continuous function in $t=0$.
}\\

We have the following.

\begin{prop}
 Let $(M,\,J)$ be a compact almost-complex manifold with
$$ H^+_J(M) \;=\; \C\left\langle \left[\alpha^1\right],\,\ldots,\,\left[\alpha^r\right]\right\rangle \;,$$
where $\alpha^1,\ldots,\alpha^r\in\wedge^{1,1}_J(M)\cap\wedge^2(M)$.
Take $L\in\textrm{End}(TM)$ and consider the curve of almost-complex structures defined by
\begin{equation}\label{eq:Jt}
J_t \;:=\; \left(\id-t\,L\right)\,J\,\left(\id-t\,L\right)^{-1}
\end{equation}
for $t$ small enough (see \cite{audin-lafontaine}),
and this representation is unique if we restrict to $L\in T^{1,0}_JM\otimes T^{*\,0,1}_JM$
(i.e., $L$ such that $LJ+JL=0$). \newline
Then, for each $\left[\alpha\right]\in H^+_J(M)$ with $\alpha\in\wedge^{1,1}_{J}(M)\cap\wedge^2M$,
there exists a real $2$-form $\eta_t=\alpha+\opiccolo{1}\in \wedge^{1,1}_{J_t}(M)\cap\wedge^2M$ such that
$\de\eta_t=0$ if and only if the following holds: there exists
$\left\{\beta_j\right\}_{j\in\N\setminus\{0\}}\subseteq \wedge^2M$ solution of the system
\begin{eqnarray}\label{eq:condizione-generale}
\de\Biggl(
\beta_j+2\,\alpha\left(L^j\sspace,\,\ssspace\right)+4\,\sum_{k=1}^{j-1}\alpha\left(L^{j-k}\sspace,
\, L^k\ssspace\right)+2\,\alpha\left(\sspace,\,L^j\ssspace\right)\\[5pt]
\nonumber
+\sum_{h=1}^{j-1}\left(2\,\beta_h\left(L^{j-h}\sspace,\,\ssspace\right)+4\,\sum_{k=1}^{j-h-1}
\alpha\left(L^{j-h-k}\sspace,\, L^k\ssspace\right)+2\,\alpha\left(\sspace,\,L^{j-h}\ssspace\right)\right)\Biggr)
\;=\;0
\end{eqnarray}
varying $j\in\N\setminus\{0\}$. In particular, the first order obstruction reads
\begin{eqnarray}\label{eq:condizione-1}
\de\left(\beta_1+2\,\alpha(L\sspace,\,\ssspace)+2\,\alpha(\sspace,\,L\ssspace)\right) \;=\;0 \;.
\end{eqnarray}
\end{prop}

\begin{proof}
Setting $J_t$ as in \eqref{eq:Jt} and expanding in series, one has
$$ J_t \;=\; J+\sum_{j\geq 1}2\,t^j\,J\,L^j $$
and then, for $\varphi\in\wedge^2M$, one computes
$$ J_t\,\varphi(\sspace,\,\ssspace) \;=\; J\,\varphi(\sspace,\,\ssspace)+2t\,J\,\left(\varphi(L\sspace,\,
\ssspace)+\varphi(\sspace,\,L\ssspace)\right)+\opiccolo{t} $$
and
$$ \de^c_{J_t}\,\varphi \;:=\; J_t^{-1}\,\de\,J_t \,\varphi\;=\; \de^c_{J}\,\varphi +2t\,J_t\,\de\,J\,
\left(\varphi(L\sspace,\,\ssspace)+\varphi(\sspace,\,L\ssspace)\right)+\opiccolo{t} \;. $$

Now, given $\alpha\in H^+_J(M)$, let $\left\{\beta_j\right\}_j$ be such that \eqref{eq:condizione-generale} holds; we put
$$ \alpha_t \;:=\; \alpha+\sum_{j\geq 1} t^j\,\beta_j $$
and
$$ \eta_t \;:=\; \frac{\alpha_t+J_t\,\alpha_t}{2} \;.$$
In this manner, we have that $\eta_t$ is a $J_t$-invariant real form such that $\eta_t=\alpha+\opiccolo{1}$.
We have to show that $\de\eta_t=0$. A straightforward computation yields
\begin{eqnarray*}
\de\eta_t &=& \sum_{j\geq1} t^j\, \de\Biggl(\beta_j+2\,\alpha\left(L^j\sspace,\,\ssspace\right)+4\,
\sum_{k=1}^{j-1}\alpha\left(L^{j-k}\sspace,\, L^k\ssspace\right)+2\,\alpha\left(\sspace,\,L^j\ssspace\right)\\[5pt]
&&+\sum_{h=1}^{j-1}\left(2\,\beta_h\left(L^{j-h}\sspace,\,\ssspace\right)+4\,\sum_{k=1}^{j-h-1}\alpha\left(L^{j-h-k}
\sspace,\, L^k\ssspace\right)+2\,\alpha\left(\sspace,\,L^{j-h}\ssspace\right)\right)
\Biggr)
\end{eqnarray*}
therefore $\de\eta_t=0$. The same computations prove the converse.
\end{proof}

\begin{rem}
 If $\de\, J_t \,=\, \pm J_t\,\de$ on $\wedge^2M$ for each $t$, then one could simply let
$$ \eta_t \;:=\; \frac{\alpha+J_t\,\alpha}{2} $$
so that $\eta_t\in\wedge^{1,1}_{J_t}$ and $\de\eta_t=0$. This is the case, for example, if all the $J_t$
are Abelian complex structures; \cite[Theorem 6]{poon} characterizes the $2$-step nilmanifolds all of
whose complex deformations are Abelian.
\end{rem}

\medskip

Finally, seeming that the result is not so useful in practice, we provide an example that employs it.

We will use the following argument, which says that, on a compact quotient of a simply-connected Lie group,
we could check the obstructions using only invariant forms.
\begin{lemma}\label{lemma:invariant}
 Let $M$ be a compact quotient of a simply-connected Lie group, so in particular $M$ is unimodular, see \cite{milnor}; if $\alpha\in\wedge^{1,1}_JM\cap \wedge^2M$ is an invariant $2$-form, if $J_t$ are invariant almost-complex structures and if we could find a solution $\left\{\beta_j\right\}_j$ to \eqref{eq:condizione-generale}, then we can find an invariant solution $\left\{\hat\beta_j\right\}_j$ to the same one.
\end{lemma}

\begin{proof}
It is enough to define
$$ \hat\beta_j \;:=\; \int_M \beta_j \,\eta $$
where $\eta$ is a bi-invariant volume form (see \cite{milnor}, \cite{belgun}).
\end{proof}

We can now give an example of a curve of compact almost-complex manifolds such that \eqref{eq:condizione-1} is not satisfied and, therefore, there is not a semi-continuity in the stronger sense described above.

\begin{ex}
 For suitable $c\in\R$, consider the solvmanifold
$$ N^6 \;:=:\; N^6(c) \;:=\; \frac{\mathrm{Sol}(3)\times\mathrm{Sol}(3)}{\Gamma(c)} $$
given in \cite{fernandez-munoz-santisteban} (see also \cite{benson-gordon}) as an example of a
cohomologically K\"ahler manifold without K\"ahler structures. (In the following, we consider $c=1$.) The structure equations of $N^6$ are
$$ \left( 12,\; 0,\; -36,\; 24,\; 56,\; 0 \right) \;.$$	
We look for a curve $\left\{J_t\right\}_t$ of almost-complex structures on $N^6$ and for a $J_0$-invariant form $\alpha$ that do not satisfy the obstruction \eqref{eq:condizione-1}: therefore, there will not be a $J_t$-invariant class close to $\alpha$. In other words, we ask for a direction $L$ along which we do not have the strong semi-continuity in the sense described above.\\
Consider the almost-complex structure
$$ J\;=\;
\left(
\begin{array}{c|c}
 \mathbf{0} & -\mathbf{1} \\
 \hline
 \mathbf{1} & \mathbf{0}
\end{array}
\right) \;;$$
let
$$ L\;=\;
\left(
\begin{array}{c|c}
 \mathbf{A} & \mathbf{B} \\
 \hline
 \mathbf{B} & -\mathbf{A}
\end{array}
\right) \;;$$
where
$$ \mathbf{A}\;=\; \left(a_{i}^{j}\right)_{i,j\in\{1,2,3\}}\;,\qquad\qquad\mathbf{B}\;=\;
\left(b_{i}^{j}\right)_{i,j\in\{1,2,3\}} $$
are constant matrices;
for
$$ \alpha \;=\; e^{14} $$
we have
\begin{eqnarray*}
\de\left(\alpha(L\sspace,\,\ssspace)+\alpha(\sspace,\,L\ssspace)\right) &=& b_1^3\, e^{123}+a_1^2\,
e^{125}-a_1^3\, e^{126}+b_1^3\, e^{136}-a_1^2\, e^{156}\\[5pt]
&&+a_1^3\, e^{234}-b_1^2\, e^{245}-b_1^3\, e^{246}+a_1^3\, e^{346}+b_1^2\, e^{456} \;.
\end{eqnarray*}
Then we choose
$$ L\;=\;
\left(
\begin{array}{ccc|ccc}
&&&&&b_1^3 \\
&&&&0& \\
&&&0&& \\
\hline
&&b_1^3&&& \\
&0&&&& \\
0&&&&&
\end{array}
\right) $$
with $b_1^3\neq0$ a constant in $\R$.\\
We need now to find a $2$-form $\beta$ such that
\begin{equation}\label{eq:ex}
\de \beta \;=\; b_1^3\, e^{123}+b_1^3\,e^{136}-b_1^3\,e^{246}\;;
\end{equation}
it is straightforward to check that no invariant $\beta$ satisfying \eqref{eq:ex} could exist;
therefore, by Lemma \ref{lemma:invariant}, also no non-invariant such $\beta$ could exist.
\end{ex}

We resume the content of the last example in the following.

\begin{prop}\label{prop:scs-forte}
 There exist a compact \Cpf\ almost-complex manifold $(M,J_0)$ and a curve $\{J_t\}_t$ of
almost-complex structures on it such that, for every $t$ small enough, there exists no $J_t$-invariant class close to
every fixed $J_0$-invariant one.
\end{prop}

\bigskip

\noindent{\sl Acknowledgments.} The authors would like to thank Tedi Dr\v{a}ghici, Tian-Jun Li and Weiyi Zhang for their very useful comments, suggestions and remarks. Also many thanks to the anonymous referee and to L\'aszl\'o Lempert for their valuable comments and remarks which improved the presentation of this paper.

\end{document}